\documentclass[a4paper,10pt]{amsart}

\usepackage{amsrefs}
\usepackage{amsfonts,mathrsfs}
\usepackage{amsthm}
\usepackage{amssymb}
\usepackage{enumerate}
\usepackage{tikz-cd}
\usepackage{cleveref}
\usepackage{enumitem}
\usepackage[obeyFinal,textsize=footnotesize]{todonotes} 
\usepackage{soul}

\theoremstyle{definition}
\newtheorem{Def}{Definition}[section]

\newtheorem{rem}[Def]{Remark}
\theoremstyle{plain}

\newtheorem{thm}[Def]{Theorem}
\newtheorem*{thm*}{Theorem}
\newtheorem{lem}[Def]{Lemma}

\newtheorem{cor}[Def]{Corollary}
\newtheorem*{cor*}{Corollary}
\newtheorem{con}[Def]{Conjecture}

\newtheorem*{con*}{Conjecture}

\newtheorem*{verm*}{Vermutung}

\newcommand{\quash}[1]{}

\usepackage{mathrsfs}

\newcommand{\BR}{\operatorname{BR}}
\newcommand{\upL}{\operatorname{L}}

\newcommand{\trop}{\operatorname{trop}}

\newcommand{\Span}{\operatorname{span}}

\newcommand{\Lo}{\mathbb{L}}

\newcommand{\cA}{{\mathcal A}}

\newcommand{\R}{{\mathbb R}}

\newcommand{\Q}{{\mathbb Q}}

\newcommand{\Z}{{\mathbb Z}}

\providecommand{\R}{\mathbb R}
\providecommand{\Cut}{\operatorname{Cut}}

\title{A note on bounded ratios}

\author{Lorenzo~Baldi}
\address{Goethe-Universität, Frankfurt am Main, Germany}
\email{baldi@math.uni-frankfurt.de}
\author{Mario Kummer}
\address{Technische Universit\"at, Dresden, Germany} 
\email{mario.kummer@tu-dresden.de}


\begin{document}
\begin{abstract}
We prove that the set of bounded ratios $\BR(X)$ on a semialgebraic set $X\subset\R^n_{>0}$ is the convex cone of linear forms that are nonnegative on the tropicalization $\trop(X)$. In particular, it is a rational polyhedral convex cone. For $X$ the set of Lorentzian polynomials with fixed M-convex support, it is the dual to the set of M-convex functions. We record an explicit counterexample to a conjecture of Huang--Huh--Soskin--Wang on the bounded ratios on Lorentzian polynomials. The bounded ratio in the counterexample corresponds to the non-hypermetric clique-web facet $\mathrm{CW}^1_7(1,1,1,1,1,-1,-1)$ of the cut cone on seven vertices. 
\end{abstract}
\maketitle
\section{Introduction}
The \emph{set of bounded ratios} on a semialgebraic subset $X \subset \R^n_{>0}$ of the positive orthant was defined in \cite{HHSW25} as
\begin{equation*}
    \BR(X)=\{ \, \lambda \in \R^n \mid \exists c \in \R_{>0} \text{ s.t. } x^\lambda \le c \text{ on } X \,\}.
\end{equation*}
{The set of bounded ratios is a closed convex cone in $\R^n$ \cite{HHSW25}*{page 2}. Given a set $A \subset \R^n$ (not necessarily convex),  we denote its \emph{dual convex cone} as 
\[
    A^\vee = \{ \, \lambda \in \R^n \mid \langle\lambda,x\rangle\geq0 \text{ for all } x\in A \, \}
\]
where$\langle\cdot,\cdot\rangle$ is the standard inner product on $\R^n$. We characterize $\BR(X)$ as the dual convex cone to $\trop(X)$, the \emph{tropicalization} of $X$ (see Section~\ref{sec:thm} for the definition).
\begin{thm}\label{thm:main}
    Let $X \subset \R^n_{>0}$ be a semialgebraic subset of the of the positive orthant. Then $\BR(X) = \trop(X)^\vee$. In particular, the set of bounded rations $\BR(X)$ is a rational polyhedral convex cone.
\end{thm}
We prove Theorem~\ref{thm:main} in \Cref{sec:thm}.
} Although this was known in several special cases and under some additional hypotheses on $X$, e.g.~\cite{blekhermanMomentsSumsSquares2025}*{Remark~2.6} or \cite{HHSW25}*{page~2}, we did not find this statement in its full generality in the literature. {In the case that $X$ is the set of Lorentzian polynomials $\upL_J$ \cite{lorentzian} with fixed M-convex monomial support $J\subset\Z_{\geq0}^n$, we obtain the following generalization of \cite{HHSW25}*{Theorem B}.
\begin{cor}\label{cor:lorentzian}
    Let $\mathrm{M}(J)$ be the set of all M-convex functions $J\to\R$. Then $$\BR(\upL_J)=\mathrm{M}(J)^\vee.$$
\end{cor}}

In \Cref{sec:counter} we provide a counterexample to a conjecture in \cite{HHSW25} about bounded ratios on the set of Lorentzian polynomials of degree $2$ with full support.
\section{The set of bounded ratios is the dual of the tropicalization}\label{sec:thm}
Let $X\subset\R^n_{>0}$ be a semialgebraic set. Its \emph{tropicalization} is defined as \cite{alessandrini}
\begin{equation*}
    \trop(X):=\lim_{q\to\infty}\underbrace{\{(-\log_q a_1,\ldots,-\log_q a_n)\mid (a_1,\ldots,a_n)\in X\}}_{=:\cA_q(X)}
\end{equation*}
where the limit on the right-hand side is taken with respect to the Hausdorff metric. One inclusion in Theorem \ref{thm:main} follows immediately from this definition.
\begin{lem}
    For all $\lambda\in\BR(X)$ and $x\in\trop(X)$ we have $\langle\lambda,x\rangle\geq0$, {i.e., $\BR(X) \subset \trop(X)^\vee$}.
\end{lem}
\begin{proof}
    Let $\lambda\in\BR(X)$. By definition, there exists $c\geq1$ such that $a^\lambda\leq c$ for all $a\in X$. 
    Taking the logarithm on both sides gives for all $q\geq1$:
    \begin{equation*}
        \lambda_1\cdot\log_qa_1+\cdots+\lambda_n\cdot\log_qa_n\leq\log_q(c).
    \end{equation*}
    This shows that $\langle\lambda,x\rangle\geq-\log_q(c)$ for all $x\in\cA_q(X)$. Since $\lim_{q\to\infty}\log_q(c)=0$, this implies the claim.
\end{proof}
For the other direction, we use the following definition of $\trop(X)$ which is equivalent to the above \cite{alessandrini}. Let $R$ be the field of real algebraic Puiseux series over $\R$. Recall that this field is isomorphic to the field of germs at $0$ of continuous semialgebraic functions \cite{BasuPollackRoy}*{Theorem~3.14}. The ordering $\le$ on $R$ is given as follows: for $a = [f], b = [g] \in R$  germs of continuous semialgebraic functions, write $a \le b$ if and only if there exists $\epsilon \in \R_{>0}$ such that $f(c) \le g(c)$ for all $c\in (0, \epsilon)$. Equivalently, a non-zero Puiseux series is positive if the sign of its leading coefficient is positive. The field $R$ has a non-trivial valuation $v$ which sends a non-zero Puiseux series to its smallest exponent. This valuation is \emph{convex} in the sense that $v(a) \ge v(b)$ if $a,b\in R$ and $ 0 \le a \le b$.
Then the tropicalization of $X \subset \R^n_{>0}$ is equal to the Euclidean closure of
\[
    \{(v(a_1),\ldots,v(a_n))\in\R^n\mid a\in X_R\}
\]
where $X_R$ is {the extension of $X$ to $R$, i.e.} the set of all points in $R_{>0}^n$ which satisfy the defining inequalities of $X$ \cite{alessandrini}*{Corollary~4.6}.
\begin{proof}[Proof of Theorem \ref{thm:main}]
    Let $K {= \trop(X)^\vee}$. We have to show that $K\subset\BR(X)$.
    By \cite{baldiNoteOminimalTropicalizations2026}*{Proposition 3.4.4}, see also \cites{allamigeon,dries}, the tropicalization $\trop(X)$ is a finite union of rational polyhedra. Thus $K$ is a rational polyhedral convex cone and it suffices to show that every $\lambda\in K\cap\Q^n$ is in $\BR(X)$. Assume for the sake of a contradiction that $\lambda \notin \mathrm{BR}(X)$. By the curve selection lemma \cite{BasuPollackRoy}*{Theorem~3.19}
    there exists a semialgebraic curve $\gamma$ inside $X$ such that $\gamma(t)^\lambda\to\infty$ as $t\to0$. Let $\alpha$ be the point in ${X_R} \subset R^n_{>0}$ associated to $\gamma$. Since $\gamma(t)^\lambda\to\infty$, we have \[0>v(\alpha^\lambda)=\lambda_1v(\alpha_1)+ \cdots +\lambda_n v(\alpha_n).\] This is a contradiction to $\lambda\in K$.
\end{proof}
\begin{rem}
    Except for the rationality statement, one can generalize Theorem~\ref{thm:main} from semialgebraic sets to definable sets in a polynomially bounded o-minimal structure, using the results from \cite{baldiNoteOminimalTropicalizations2026}.
\end{rem}
\begin{proof}[Proof of Corollary \ref{cor:lorentzian}]
    This follows immediately from Theorem \ref{thm:main} and \cite{lorentzian}*{Theorem 3.20} which says that $\trop(\upL_J)=\mathrm{M}(J)$.
\end{proof}
\section{A counterexample to the bounded ratio conjecture}\label{sec:counter}
A \emph{Lorentzian matrix} is a real symmetric matrix with nonnegative entries and at most one positive eigenvalue.  We denote by $\Lo_n^+$ the set of $n\times n$ Lorentzian matrices with all entries positive.  Huang, Huh, Soskin, and Wang study studied $\BR(\Lo_n^+)$ in \cite{HHSW25}.
By conjugating with diagonal matrices with positive entries, which preserves being Lorentzian, the diagonal entries of a matrix in $\Lo_n^+$ can be made equal to one. We call such a matrix \emph{normalized}.  A \emph{reduced exponent vector} is a vector
\[
        \alpha=(\alpha_{ij})_{1\leq i<j\leq n}\in\R^{\binom n2}.
\]
It is called a \emph{reduced bounded ratio} on $\Lo_n^+$ if there is a constant $c>0$ such that
\[
        P^\alpha:=\prod_{1\leq i<j\leq n}p_{ij}^{\alpha_{ij}}\leq c
\]
for all normalized matrices $P=(p_{ij})\in\Lo_n^+$.  The infimum of all such constants is denoted by $f(\alpha)$.  Following \cite{HHSW25}*{Definition 1.8}, a nonzero reduced bounded ratio is called \emph{normalized} if
\[
        \sum_{1\leq i<j\leq n}\alpha_{ij}=-1.
\]
We will disprove the following conjecture.

\begin{con}[\cite{HHSW25}*{Conjecture 1.9}]\label{con:br}
For every $n$ and every normalized reduced ratio $\alpha$ on on $\Lo_n^+$ one has $f(\alpha)\leq2$.
\end{con}

We will use the description of bounded ratios in terms of the cut cone.  For $S\subseteq[n]$, let $\delta(S)\in\R^{\binom n2}$ be the cut vector
\[
        \delta(S)_{ij}=\begin{cases}
        1,&\text{if }S\text{ contains exactly one of }i,j,\\
        0,&\text{otherwise.}
        \end{cases}
\]
The \emph{cut cone} $\Cut_n$ is the cone generated by the vectors $\delta(S)$. The set $\underline{\operatorname{BR}}(\Lo_n^+)$ of reduced bounded ratios, defined as $\BR(\Lo_n^+ \cap \{ \, P =(p_{ij}) \colon p_{ii}=1 \text{ for }i=1, \dots , n \, \})$ before, is equal to minus the dual convex cone of $\Cut_n$ by \cite{HHSW25}*{Theorem B}:
\[
        \underline{\operatorname{BR}}(\Lo_n^+)
        = - \Cut_n^\vee.
\]
Since $f$ is log-convex, Conjecture \ref{con:br} can be reduced to extreme rays of $\underline{\operatorname{BR}}(\Lo_n^+)$, which correspond to facets of $\Cut_n$; see the discussion following \cite{HHSW25}*{Conjecture 1.9}. 

Conjecture \ref{con:br} is true for $n\leq5$ by \cite{HHSW25}*{Theorem A}. According to \cite{DL10}*{Remark 15.2.11} there are, up to symmetry and in addition to the inequalities on $5\times5$ submatrices, two new inequalities needed to describe $\Cut_6$. These are of a certain type that is called \emph{hypermetric} and we were not able to produce a counter-example for them. A complete list of facets of $\Cut_7$ is given in \cite{DL10}*{Section 30.6}. We use the first inequality from this list which is not hypermetric:
\[
        \mathrm{CW}^1_7(1,1,1,1,1,-1,-1)^Tx\leq0.
\]
By \cite{DL10}*{Definition 29.1.2} this inequality is given by
\begin{equation}\label{eq:facet}
z:=        x_{13}+x_{14}+x_{24}+x_{25}+x_{35}+x_{67}
        -\sum_{i=1}^5(x_{i6}+x_{i7})\leq0.
\end{equation}
Since the coordinate sum of $z$ is $-4$, the vector $\alpha:=\frac{1}{4}z$ is normalized. 

Now consider the three-parameter family
\begin{equation*}\label{eq:Prqt}
P(r,q,t)=
\begin{pmatrix}
1&r&q&q&r&t&t\\
r&1&r&q&q&t&t\\
q&r&1&r&q&t&t\\
q&q&r&1&r&t&t\\
r&q&q&r&1&t&t\\
t&t&t&t&t&1&1\\
t&t&t&t&t&1&1
\end{pmatrix}.
\end{equation*}
This shape is adapted to the support of \eqref{eq:facet}. Namely, the five positive summands
\[
        x_{13},\ x_{14},\ x_{24},\ x_{25},\ x_{35}
\]
are equal to $q$, the summand $x_{67}$ is equal to $1$, and all ten negative summands $x_{i6},x_{i7}$, $1\leq i\leq5$, are equal to $t$.  Therefore
\begin{equation}\label{eq:ratio-simplifies}
        P(r,q,t)^\alpha
        =
        \left(\frac{q^5}{t^{10}}\right)^{1/4}.
\end{equation}

\begin{lem}\label{lem:lorentzian-condition}
Assume $r,q,t>0$.  The matrix $P(r,q,t)$ is Lorentzian if and only if
\begin{equation}\label{eq:lorentzian-conditions}
        r+q-2\geq \sqrt5\,|r-q|
        \qquad\text{and}\qquad
        t^2\geq\frac{1+2r+2q}{5}.
\end{equation}
\end{lem}

\begin{proof}
We consider the decomposition $\R^7=U\oplus V\oplus W$ into the linear subspaces
\begin{align*}
    U&=\Span\{(0,\ldots,0,1,-1 )^T\},\\
    V&=\Span\{(0,\ldots,0,1,1 )^T,(1,\ldots,1,0,0 )^T\},\\
    W&=\Span\{(v_1,\ldots,v_5,0,0 )^T\mid\sum_{i=1}^5v_i=0\}.
\end{align*}
We note that these are pairwise orthogonal with respect to the bilinear form $P$, and that $U$ is in the kernel of $P$. The restriction of $P$ to $V$ is
\[
        \begin{pmatrix}
        4&10t\\
        10t&5+10q+10r
        \end{pmatrix}.
\]
This matrix has at least one positive eigenvalue and its determinant is non-positive if and only if the second inequality in (\ref{eq:lorentzian-conditions}) holds. It remains to show that the restriction of $P$ to $W$ is negative semidefinite if and only if the first inequality in (\ref{eq:lorentzian-conditions}) holds. We calculate that the eigenvalues of the upper-left $5\times5$ block of $P$ are:
\begin{align*}
    \lambda&=1+2r+2q,\\
    \mu_\pm&=\frac {1} {2}\left(\pm\sqrt {5}|r-q| - q - r + 
   2 \right).
\end{align*}
This can be done brute-force or by making use of the fact that it is a \emph{circulant matrix}. For such matrices there exist formulas to calculate their eigenvalues \cite{circulant}. We further note that $\lambda$ is the eigenvalue of the all-ones vector and the eigenvalues $\mu_\pm$ have multiplicity $2$. Hence $P$ is negative semidefinite on $W$ if and only if $\mu_\pm\leq0$. This is equivalent to the first inequality in (\ref{eq:lorentzian-conditions}).
\end{proof}

The conditions in Lemma \ref{lem:lorentzian-condition}, together with the requirement that the value in \eqref{eq:ratio-simplifies} is larger than $2$, are
\[
        r+q-2\geq \sqrt5|r-q|,
        \qquad
        5t^2\geq1+2r+2q,
        \qquad
        q^5>16t^{10}.
\]
These inequalities are all satisfied for
\[
        r=284,
        \qquad
        q=701,
        \qquad
        t=20.
\]
Evaluating the normalized bounded ratio (\ref{eq:ratio-simplifies}) gives
\[
        \left(\frac{q^5}{t^{10}}\right)^{1/4}
        =
        \left(\frac{701}{400}\right)^{5/4}
        \approx 2.01638>2
\]
and hence Conjecture~\ref{con:br} is disproved.

\section*{Acknowledgments}
L.~Baldi thanks R. Sinn and J. Weigert for introducing to him the concept of bounded ratios, and for helpful discussions. L.~Baldi was partially funded by the Deutsche Forschungsgemeinschaft (DFG, German Research Foundation) – Project number 582983756.
\bibliographystyle{plain}
\bibliography{biblio.bib}

\end{document}